\documentclass[11pt]{amsart}
\usepackage{amsmath,amsthm,amssymb,mathrsfs}
\usepackage{microtype}
\usepackage{xspace}
\usepackage{cleveref,cite,url}
\usepackage{graphicx}
\usepackage[utf8]{inputenc}
\usepackage{enumitem}
\usepackage{dsfont} %needed for \mathsd{1}

\theoremstyle{plain}% default
\newtheorem{theorem}{Theorem}
\newtheorem{lemma}{Lemma}
\newtheorem{proposition}{Proposition}

\theoremstyle{definition}

\newtheorem{example}{Example}

\theoremstyle{remark}
\newtheorem*{remark}{Remark}

\def\addQED{}
\def\proofmain{Proof of \Cref{thm:sufficient}}
\def\proofmainconic{Proof of \Cref{thm:sufficientconic}}
\def\prooflicq{Proof of \Cref{thm:licqgeneric}}

\newcommand{\ZZ}{\mathbb{Z}}

\newcommand{\NN}{\mathbb{N}}
\newcommand{\RR}{\mathbb{R}}

\renewcommand{\SS}{\mathbb{S}}
\newcommand{\id}{I}

\DeclareMathOperator{\rank}{rank}

\DeclareMathOperator{\image}{Im}
\DeclareMathOperator{\spann}{span}
\newcommand\mydots{\hbox to 0.9em{.\hss.\hss.}\kern.05em}

\begin{document}

% \title[Burer-Monteiro guarantees for general SDPs]{Burer-Monteiro guarantees for \\general semidefinite programs}
\title[On the Burer-Monteiro method for general SDPs]{On the Burer-Monteiro method\\ for general semidefinite programs}

\author{Diego Cifuentes}
\address{Massachusetts Institute of Technology \\ Cambridge, MA, USA}
\email{diegcif@mit.edu}

\begin{abstract}
  Consider a semidefinite program (SDP) involving an $n\times\nobreak n$ positive semidefinite matrix~$X$.
The Burer-Monteiro method uses the substitution $X=Y Y^T$
to obtain a nonconvex optimization problem in terms of an $n\times p$ matrix~$Y$.
Boumal et al.\ showed that this nonconvex method provably solves equality-constrained SDPs with a generic cost matrix when $p \gtrsim\nobreak \sqrt{2m}$,
where $m$ is the number of constraints.
In this note we extend their result to arbitrary SDPs, possibly involving inequalities or multiple semidefinite constraints.
We derive similar guarantees for a fixed cost matrix and generic constraints.
We illustrate applications to matrix sensing and integer quadratic minimization.

\end{abstract}

\maketitle

\section{Introduction} \label{sec:introduction}

Consider a \emph{semidefinite program} (SDP) in $\SS^n$, the space of $n{\times} n$ symmetric matrices, with $m\!=\!m_1\!+\!m_2$ constraints
($m_1$~equalities and $m_2$ inequalities):
\begin{equation} \label{eq:sdp}
  \tag{\textit{SDP}}
  \begin{gathered}
    \min_{X\in \mathscr{X}}\; C\bullet X, \qquad
    \mathscr{X} \,:=\, \{X \in \SS_+^n:
  \mathcal{A}(X) - b \in \{0\}^{m_1}{\times} \RR_+^{m_2} \}
  \end{gathered}
\end{equation}
where $C\!\in\! \SS^n$, $b\!\in\!\RR^m$
and $\mathcal{A}:\SS^n\!\to\!\RR^m$,
$X \!\mapsto\! (A_1\!\bullet\! X,\mydots, A_m\!\bullet\! X)$ is a linear map.
We assume that $\mathscr{X}$ is nonempty and that the minimum is achieved.
Though interior point methods can solve~\eqref{eq:sdp} in polynomial time,
they often run into memory problems for large values of~$n$.
This has motivated a surge of newer, more scalable techniques;
see the recent survey~\cite{Majumdar2019}.
We study here the low rank factorization method, pioneered by Burer and Monteiro~\cite{Burer2003,Burer2005}.

The Burer-Monteiro method consists in writing $X= Y Y^T$ for some $Y\in \RR^{n\times p}$,
and solving the following nonconvex optimization problem:
\begin{equation} \label{eq:bm}
  \tag{\textit{BM}}
  \begin{gathered}
    \min_{Y\in \RR^{n\times p}}\;  C\bullet Y Y^T \quad\text{ such that }\quad Y Y^T \in \mathscr{X}.
  \end{gathered}
\end{equation}
Let $\tau(k):= \binom{k+1}{2}$ be the $k$-th triangular number.
Barvinok~\cite{Barvinok1995problems} and Pataki~\cite{Pataki1998} independently showed that \eqref{eq:sdp} has an optimal solution of rank $r$, with $\tau(r) \!\leq\! m$.
Consequently, problems~\eqref{eq:sdp} and~\eqref{eq:bm} are equivalent for any $p$ with $\tau(p) \!\geq\! m$.
But due to nonconvexity, local optimization methods may not always recover the global optimum of~\eqref{eq:bm}.
Nonetheless, the Burer-Monteiro performs very well in several applications, see e.g.,~\cite{Burer2003, Journee2010, Rosen2016}.

There has been much recent work in proving global guarantees for~\eqref{eq:bm}.
Most remarkably, Boumal et al.~\cite{Boumal2016,Boumal2018} showed that equality-constrained SDPs ($m_2{=}0$) have no spurious 2nd-order critical points when $\tau(p)\!>\!m$ under certain assumptions.
Concretely, they require that the cost matrix $C$ is generic and that the feasible set of \eqref{eq:bm} is sufficiently regular.
By \emph{generic} we mean that the result holds outside a set of measure zero.
Though other global guarantees for~\eqref{eq:bm} exist, e.g.~\cite{Ge2016,Park2017}, their setting is more restrictive.

In this note we generalize the result from Boumal et al.~\cite{Boumal2016,Boumal2018} to arbitrary SDPs,
possibly involving inequalities or multiple positive semidefinite (PSD) constraints.
For the inequality-constrained problem~\eqref{eq:sdp},
we show in \Cref{thm:main} that if $\tau(p)\!>\! m$ and the cost is generic, then any 2nd-order critical point of~\eqref{eq:bm} is globally optimal.
Similar guarantees might be derived even when the cost matrix is fixed, see \Cref{thm:genericA}.
We show applications to integer quadratic minimization and PSD matrix sensing.

Our proof of \Cref{thm:main} is simpler than the one in~\cite{Boumal2016,Boumal2018},
as it relies on nonlinear programming instead of Riemannian optimization.
This simplicity is reflected in the fact that \Cref{thm:main} does not require any regularity assumptions on the domain (constraint qualifications).
Nevertheless, regularity conditions might still be needed to prevent the existence of local minima that do not satisfy the 2nd-order criticality conditions.

We also consider SDPs involving multiple PSD variables,
and study the Burer-Monteiro method applied to a subset of these variables.
We prove in \Cref{thm:mainconic} that, for a generic cost, any 2nd-order critical point is globally optimal when $p$ satisfies a bound due to Pataki~\cite{Pataki1998}.
We present an application to symmetric matrix sensing
(the restricted isometry property is not needed).

The structure of this note is as follows.
\Cref{sec:criticality} reviews the notion of 2nd-order critical points in nonlinear programming.
\Cref{sec:inequality_constrained_sdp_s} analyzes the Burer-Monteiro method for the inequality-constrained problem~\eqref{eq:sdp}.
\Cref{sec:general_semidefinite_program} studies SDPs involving multiple PSD constraints.

\smallskip
\textbf{Related work.}
The guarantees from Boumal et al.\ have been further studied in~\cite{Waldspurger2018,Bhojanapalli2018,Pumir2018,Cifuentes2019poly},
but all these papers focus on the equality-constrained case.
The bound $\tau(p) > m$ was shown to be optimal up to lower order terms in~\cite{Waldspurger2018}.
Guarantees for approximate 2-critical points were derived in~\cite{Bhojanapalli2018,Pumir2018,Cifuentes2019poly}.
The first polynomial time bounds for the Burer-Monteiro method were recently proved in~\cite{Cifuentes2019poly}.
We hope that the techniques developed in this paper may lead to polynomial time guarantees for arbitrary SDPs.

\section{Criticality conditions}\label{sec:criticality}

We review the notion of critical points.
Consider the \emph{nonlinear program}
\begin{align} \label{eq:nlp} \tag{\it NLP}
\min_{y} \{ f(y) : h(y) \!\nobreak\in\nobreak\! \{0\}^{m_1}\!{\times}\RR_+^{m_2}\}.
\end{align}
Let
$L(y,\lambda) \!=\! f(y) \!-\! \lambda\!\cdot\! h(y) $
be the Lagrangian function.
Let $I(y) \!\subset\![m]$ be the indices of the \emph{active} constraints at~$y$,
i.e., the indices for which $h_i(y) \!=\! 0$.
The 1st-order and 2nd-order necessary optimality conditions are:
\begin{subequations}\label{eq:conditions0}
\begin{gather}\label{eq:conditionfirst0}
   y \text{ feasible},
  \quad
  \lambda \in \RR^{m_1}\!{\times}\RR_+^{m_2},
  \quad
  \lambda_i \!=\! 0 \text{ for } i\!\notin\! I( y),
  \quad
  \nabla_{y} L( y, \lambda) \!=\! 0,
  \\
  \label{eq:conditionsecond0}
  u^T \nabla^2_{yy} L( y, \lambda) \, u \geq 0,
  \quad\forall\, u \;\text{ such that }\;
  \nabla_y h_i(y) u \!=\! 0 \text { for } i\!\in\! I(y).
\end{gather}
\end{subequations}
A point $y$ is \emph{1st-order critical} for~\eqref{eq:nlp},
abbreviated \emph{1-critical},
if there exist multipliers~$\lambda$ satisfying~\eqref{eq:conditionfirst0}.
The point is \emph{2nd-order critical},
abbreviated \emph{2-critical},
if~\eqref{eq:conditionsecond0} also holds.
A critical point is \emph{spurious} if it is not the global minimum of~\eqref{eq:nlp}.

Given a {local minimum}~$y$ of~\eqref{eq:nlp},
it is known that $y$ satisfies~\eqref{eq:conditions0} under suitable \emph{regularity} assumptions.
Different regularity conditions, known as \emph{constraint qualifications}, have been proposed~\cite{Bazaraa2013}.
One of the simplest is:
\begin{align} \label{eq:licq} \tag{\it LICQ}
  \{\nabla h_i(y) : i \in I(y) \} \text{ are linearly independent}.
\end{align}
Various algorithms with provable convergence guarantees to 2-critical points are known,
see e.g., \cite{Andreani2010,Gill2017,Birgin2018} and the references therein.
These results rely either on \eqref{eq:licq} or a weaker constraint qualification.

More generally, consider the \emph{nonlinear conic program}
\begin{align} \label{eq:nlpconic} \tag{\it NLCP}
\min_{x,y} \{ f(x,y) : h(x,y) {=} 0, \, x{\in} \mathcal{K}\},
\end{align}
where $\mathcal{K}$ is a closed convex cone.
The Lagrangian function is
$ L(x,y,\lambda,s) = f(x,y) - \lambda {\cdot}\nobreak h(x,y) - s{\cdot} x$.
The following 1st-order conditions are necessary for optimality under suitable regularity conditions, see e.g.,~\cite[\S3.1]{Bonnans2013}:
\begin{subequations}\label{eq:conditionsconic0}
\begin{align}\label{eq:conditionconicfirst0}
  (x,y) \text{ feasible},
  \quad
  s\in \mathcal{K}^*,
  \quad
  \langle s, x\rangle = 0,
  \quad
  \nabla_{x,y} L(x, y, \lambda, s) = 0,
\end{align}
where $\mathcal{K}^*$ is the dual cone of~$\mathcal{K}$.
Though there exist 2nd-order conditions for conic programs,
it suffices for us to restrict the domain to pairs $(x,y)$ with a fixed value of $x$.
We get a nonlinear program in~$y$, with 2nd-order condition:
\begin{align}\label{eq:conditionconicsecond0}
  u^T \nabla^2_{yy} L( x,  y, \lambda,  s) \, u \geq 0,
  \quad\forall\, u \;\text{ such that }
  \nabla_y h( x,  y) \, u = 0.
\end{align}
\end{subequations}
A point $(x,y)$ is \emph{1-critical} for~\eqref{eq:nlpconic}
if it satisfies~\eqref{eq:conditionconicfirst0} for some $\lambda,s$.
The point is \emph{2-critical} if~\eqref{eq:conditionconicsecond0} also holds.

There are several algorithms for~\eqref{eq:nlpconic} for the case $\mathcal{K} \!=\! \SS_+^n$, see the survey paper~\cite{Yamashita2015}.
Symmetric cones (e.g., products of PSD cones) were studied in~\cite{Liu2008}.
These methods are provably convergent to 1-critical points.
In order to escape from points that do not satisfy~\eqref{eq:conditionconicsecond0}
we may rely on 2nd-order methods for the \eqref{eq:nlp} given by fixing the $x$~coordinate.

\section{Inequality constrained SDPs} \label{sec:inequality_constrained_sdp_s}

Consider problems~\eqref{eq:sdp} and~\eqref{eq:bm}.
For $X\!\in\! \mathscr{X}$, recall that the $i$-th constraint is {active} at~$X$ if $A_i \!\bullet\! X \!=\! b_i$.
Let $m' \!=\! m'(\mathscr{X})$ be the largest number of linearly independent constraints that can be simultaneously active.
For instance, if $m_2\!=\!0$ then $m' \!=\! \rank \mathcal{A}$.
We will show the following theorem.

\begin{theorem} \label{thm:main}
  Let $p$ such that $\tau(p) > m'$.
  For a generic~$C$,
  problem~\eqref{eq:bm} has no spurious 2-critical points.
  This means that any 2-critical point~$Y$ for~\eqref{eq:bm} is also globally optimal,
  and hence $ Y  Y^T$ is optimal for~\eqref{eq:sdp}.
\end{theorem}

\begin{example}[Integer quadratic minimization]
  Consider the optimization problem
  $\min \{ f(x) : x \!\in\! \ZZ^n \}$
  where $f(x)$ is a convex quadratic function.
  Denoting $\tilde x \!:=\! (x,1) \!\in\! \ZZ^{n+1}$,
  we may write $f(x) \!=\! \tilde{x}^T C \tilde{x}$ for some $C \!\in\! \SS^{n+1}$.
  The following SDP relaxation for this problem was proposed in~\cite{Park2018}:
  \begin{align*}
    \min_{X} \; C \!\bullet\! X
    \quad\text{ s.t. }\quad
    X_{i,i} \!\geq\!  X_{i,n+1} \text{ for } i \!\in\! [n], \quad
    X_{n+1,n+1} \!=\! 1, \quad
    X\in\SS_+^{n+1}.
  \end{align*}
  By Theorem 1, for a generic cost function any 2-critical point of the Burer-Monteiro problem is globally optimal when $\tau(p) \!>\! n{+}1$.
\end{example}

By generic, we mean the following.
For fixed $\mathcal{A}$, $b$, the set of all cost matrices $C \!\in\! \SS^n$ for which~\eqref{eq:bm} has a spurious 2-critical point has measure zero.
We can provide an explicit characterization of this measure-zero set in~$\SS^n$.
This set is contained in the Minkowski sum of two special algebraic sets.
The first algebraic set is given by a rank constraint:
\begin{gather} \label{eq:varietyrank}
  \SS^n_{n-p} \; := \;
  \{ X: \rank X \leq n{-}p\} \;\subset\; \SS^n.
\end{gather}
It is known that $\dim \SS^n_{n-p} = \tau(n)\!-\!\tau(p)$, see e.g.,~\cite[Prop.2.1]{Helmke1995}.
The second algebraic set is a union of linear subspaces:
\begin{gather} \label{eq:varietyL}
  \mathcal{L} \; := \;
  \bigcup\nolimits_I
  \mathcal{L}_I
  \;\subset\; \SS^n,
  \quad\text{ with }\quad
  \mathcal{L}_I := \spann\{ A_i : i \in I \},
\end{gather}
where the union is over the possible subsets of constraints $I\subset[m]$ that can be simultaneously active.
Note that $\dim \mathcal{L} = m'$ by definition of~$m'$.

\begin{theorem}\label{thm:sufficient}
  If~\eqref{eq:bm} has a spurious 2-critical point then
  $C \in \SS^n_{n{-}p} + \mathcal{L}$.
\end{theorem}

\Cref{thm:main} follows directly from \Cref{thm:sufficient}.
Indeed, if $\tau(p) \!>\! m'$ then
\begin{align}\label{eq:dimension}
  \dim (\SS^n_{n-r} \!{+} \mathcal{L})
  \,\leq\, (\tau(n){-}\tau(p)) + m'
  \,<\, \tau(n) = \dim \SS^n.
\end{align}
Therefore $\SS^n_{n-r} \!+\! \mathcal{L}$ is a proper algebraic set in~$\SS^n$,
and has measure zero.

We proceed to prove \Cref{thm:sufficient}.
We first derive the criticality conditions for~\eqref{eq:bm}.
This is a special instance of~\eqref{eq:nlp},
so we need to specialize~\eqref{eq:conditions0}.
We have $h(Y) \!=\! \mathcal{A}(Y Y^T) {-}b$
and $L(Y,\lambda) \!=\! S(\lambda) {\bullet} Y Y^T \!+\! b^T \lambda$, where
\begin{align*}
  S(\lambda) := C - \mathcal{A}^*(\lambda) \in \SS^n
  \,\text{ is the \emph{slack} matrix,}
\end{align*}
and $\mathcal{A}^*:\RR^m \!\to\! \SS^n$, $\lambda \!\mapsto\!\sum_i \lambda_i A_i$ is the adjoint of~$\mathcal{A}$.
The 1st-order and 2nd-order criticality conditions are:
\begin{subequations}\label{eq:conditions}
\begin{gather}
  \label{eq:conditionfirst}
  Y  Y^T \in \mathscr{X},
  \quad
  \lambda \in \RR^{m_1}\!{\times}\RR_+^{m_2},
  \quad
  \lambda_i \!=\! 0 \text{ for } i \!\notin\! I( Y),
  \quad
  S(\lambda)  Y = 0,
  \\
  \label{eq:conditionsecond}
  S(\lambda)\bullet U U^T \geq 0,
  \quad\forall\, U\!\in\! \RR^{n\times p} \text{ such that }
  A_i \bullet U  Y^T \!=0 \text { for } i\!\in\! I(Y).
\end{gather}
\end{subequations}

The following lemma establishes sufficient conditions for a critical point to be global optimal.
The lemma is known, see~\cite{Burer2005,Journee2010,Boumal2018},
but our assumptions are slightly different since we allow inequalities.

\begin{lemma} \label{thm:firstpart}
  Either of the following conditions imply global optimality:
  \begin{enumerate}[label=(\roman*)]
    \item
      $Y$ is 1-critical and the multiplier $\lambda$ satisfies $S(\lambda) \in \SS_+^n$,
    \item
      or $Y$ is 2-critical and $Y$ is column rank deficient.
  \end{enumerate}
\end{lemma}
\begin{proof}
  \textit{(i)}
  The conic dual of~\eqref{eq:sdp} is
  $
  \,\max_{\lambda}\{ b^T \lambda :
    S(\lambda) \!\in\! \SS_+^n,\,
  \lambda\!\in\! \RR^{m_1}\!{\times} \RR_+^{m_2} \}.
  $
  Let $(Y,\lambda)$ satisfy \eqref{eq:conditionfirst}, and let $X := Y Y^T$.
  We will show that the primal/dual pair $(X, \lambda)$ is optimal for the SDP.
  It suffices to verify three conditions:
  $X$ is primal feasible,
  $\lambda$ is dual feasible,
  and complementary slackness holds
  (i.e., $\lambda_i{=}0$ for $i{\notin} I( X)$ and $S(\lambda)  X {=} 0$).
  Primal feasibility and complementary slackness follow from~\eqref{eq:conditionfirst},
  while dual feasibility corresponds to $S(\lambda) \!\in\! \SS_+^n$.

  \textit{(ii)}
  Let $(Y,\lambda)$ satisfy \eqref{eq:conditions}.
  By the above item, it suffices to show that $S(\lambda) \!\in\! \SS_+^n$.
  Let $x\!\in\! \RR^n$, and let us see that $x^T S(\lambda) x \!\geq\! 0$.
  Since $Y$ is rank deficient,
  there is a nonzero vector $z\!\in\! \RR^p$ such that $Y z \!=\! 0$.
  The matrix $U \!:=\! x z^T$ satisfies $U Y^T \!\!=\! 0$,
  so $S(\lambda) \!\bullet\! U U^T \!\!\geq\! 0$ by~\eqref{eq:conditionsecond}.
  Since $S(\lambda)\!\bullet\! U U^T \!=\! \|z\|^2 (x^T S(\lambda) x)$,
  then $x^T S(\lambda) x \geq 0$.
  \addQED
\end{proof}

We are ready to prove \Cref{thm:sufficient} (which implies \Cref{thm:main}).

\begin{proof}[\proofmain]
  Let $(Y,\lambda)$ a spurious point satisfying~\eqref{eq:conditions}.
  \Cref{thm:firstpart}\textit{(ii)} gives that $\rank Y\!=\!p$.
  By \eqref{eq:conditionfirst} we have $S(\lambda)  Y \!=\! 0$, which implies $S(\lambda)\!\in\! \SS^n_{n-p}$,
  and also $\lambda_i{=}0$ for $i\!\notin\!I(Y)$.
  Thus
  $C = S(\lambda) \!+\! \mathcal{A}^*(\lambda) \in \SS^n_{n-p} \!+ \mathcal{L}$.
  \addQED
\end{proof}

To finish this section, we observe that \Cref{thm:sufficient} can be used even if the cost matrix~$C$ is not generic.
For instance, the next theorem assumes that both $b,C$ are fixed and $\mathcal{A}$ is generic (i.e., $A_1,\dots,A_m$ are generic).

\begin{theorem} \label{thm:genericA}
  Let $p$ such that $\tau(p) \!>\! m$ and $\rank C \!>\! n{-}p$.
  For a generic~$\mathcal{A}$,
  problem~\eqref{eq:bm} has no spurious 2-critical points.
\end{theorem}
\begin{proof}
  By \Cref{thm:sufficient}, it suffices to see that $C \notin \SS^n_{n-p} \!+\! \mathcal{L}$.
  Fix $I\!\subset\! [m]$, and let $\mathcal{L}_I\!\subset\nobreak\! \SS^n$ as in~\eqref{eq:varietyL}.
  Note that $\mathcal{L}_I$ is generic among the subspaces of dimension~$|I|$,
  as it depends on the generic matrices~$A_i$.
  Recall that $\dim( \SS^n_{n-p} {+} \mathcal{L}_I) \!<\! \dim\SS^n$ by~\eqref{eq:dimension}.
  Since $C\!\notin\! \SS^n_{n-p}$, then $C \notin \SS^n_{n-p} \!+\! \mathcal{L}_I$ for a generic~$\mathcal{L}_I$.
  The result follows from $\mathcal{L} = \bigcup_I \mathcal{L}_I$.
  \addQED
\end{proof}

An additional advantage of having generic constraints is that regularity is always satisfied.
Therefore any local minimum of~\eqref{eq:bm} is also 2-critical,
and hence is subject to \Cref{thm:genericA}.
The next proposition is shown in \Cref{sec:regularity}.

\begin{proposition} \label{thm:licqgeneric}
  Assume that the entries of $b$ are nonzero.
  For a generic $\mathcal{A}$,
  any feasible point of \eqref{eq:bm} satisfies \eqref{eq:licq}.
\end{proposition}

\begin{example}[Matrix sensing] \label{exmp:sensing}
  Given a linear map $\mathcal{A}:\SS^n \!\to\! \RR^m$ and a vector $b \!\in\! \RR^m$,
  consider finding a low rank matrix $X \!\in\! \SS^n$ such that $\mathcal{A}(X) \!=\! b$.
  A~standard technique to promote low rank is to minimize the nuclear norm:
  \begin{align}\label{eq:sensing}
    \min_{X\in\SS^n} \; \|X\|_*
    \quad\text{ such that }\quad \mathcal{A}(X)=b.
  \end{align}
  If we further assume that $X$ that is PSD, the cost function is $\id_n \bullet X$.
  By \Cref{thm:genericA}, if $\mathcal{A}$ is generic and $\tau(p)\!>\!m$,
  then any local minimum of~\eqref{eq:bm} is globally optimal.
  The PSD assumption will be relaxed in the next section.
\end{example}
\begin{remark}
  Different guarantees about the Burer-Monteiro method for matrix sensing were obtained in~\cite{Park2017}, relying on the restricted isometry property.
\end{remark}

\section{General SDPs} \label{sec:general_semidefinite_program}

Let $\textbf{n} \!:=\! (n_1,\dots,n_\ell)\in \NN^\ell$ and $d\in \NN$.
We consider an SDP involving PSD matrices of sizes $n_1,\dots,n_\ell$ and a free variable of dimension~$d$.
Let the Euclidean space
$\SS^{\bf n} := \SS^{n_1}{\times} \cdots {\times} \SS^{n_\ell}$
and the convex cone
$\SS^{\bf n}_+ := \SS_+^{n_1}{\times} \cdots {\times} \SS_+^{n_\ell}$.
Given $C \in \SS^{\bf n}{\times}\RR^d$, $b \in \RR^m$, and a linear map $\mathcal{A}:\SS^{\bf n}{\times}\RR^d\!\to\! \RR^{m}$,
consider:
\begin{equation} \label{eq:sdpconic}
  \tag{${\mathit{SDP}}_{\!\bf n}$}
  \begin{gathered}
    \min_{X\in \mathscr{X}}\; \langle C,  X\rangle,
  \qquad
  \mathscr{X} := \{X\!\in \SS^{\bf n}_+\!\times\!\RR^d \,:\,
  \mathcal{A}(X) \!=\! b \},
  \end{gathered}
\end{equation}
where $X := (X_1,\dots,X_\ell,x)$ with $X_j\in \SS^{n_j}$, $x\in\RR^d$.
As before, we assume that $\mathscr{X}$ is nonempty and that the minimum is achieved.

We apply the Burer-Monteiro method to the first $k$ matrices.
Let $Y := (Y_1,\mydots, Y_k)$, with $Y_j \!\in\! \RR^{n_j\times p_j}$, and let $q(Y):= (Y_1 Y_1^T, \mydots, Y_k Y_k^T)$.
We denote
\begin{align*}
\underline{\bf n} \!:=\! (n_1,\mydots, n_k), \quad
\overline{\bf n} \!:=\! (n_{k+1},\mydots,n_\ell), \quad
\overline{X} := (X_{k+1},\mydots,X_\ell).
\end{align*}
In particular, $\SS^{\bf n} = \SS^{\underline{\bf n}} \times \SS^{\overline{\bf n}}$.
The Burer-Monteiro problem is:
\begin{equation} \label{eq:bmconic}
  \tag{${\mathit{BM}}_{\!\bf n}$}
  \begin{gathered}
    \min_{Y,\,\overline{X},\, x}\;\;
    \langle\, C\,,\, (q(Y), \overline{X}, x)\,\rangle
    \quad\text{ such that }\quad
    (q(Y), \overline{X}, x) \,\in\, \mathscr{X}.
  \end{gathered}
\end{equation}

Pataki~\cite{Pataki1998} showed that~\eqref{eq:sdpconic} always has an optimal solution such that $\sum_{j=1}^\ell \tau(r_j) \!\leq\! m \!-\! d$, where $r_j \!:=\! \rank X_j$.
We can ensure that there is a solution with $r_j \!\leq\! p_j$ for all $j \!\in\! [k]$
if either $p_j \!\geq\! n_j$ or $\tau(p_j) \!\geq\! m'$, with
\begin{align*}
  m' \;:=\; \max_{r_{k+1},\dots, r_\ell} \quad  m - d - \tau(r_{k+1}) - \tau(r_{k+2}) - \,\cdots\, -\tau(r_\ell),
\end{align*}
where the maximum is over the possible ranks $r_{k+1},\dots,r_\ell$.
Hence, problems~\eqref{eq:sdpconic} and~\eqref{eq:bmconic} agree when
$\tau(p_j) \!\geq\! \min\{m', \tau(n_j)\}$ for $j \!\in\! [k]$.

\begin{theorem} \label{thm:mainconic}
  Assume that $\tau(p_j) \!>\! \min\{m',\tau(n_j)\}$ for $j \!\in\! [k]$.
  For a generic~$C$, problem~\eqref{eq:bmconic} has no spurious 2-critical points.
\end{theorem}

\begin{example}[Inequalities]
  Consider the inequality constrained problem \eqref{eq:sdp}.
  We may view each of the $m_2$ inequalities as a PSD constraint on a $1{\times} 1$ matrix.
  So this is a special instance of \eqref{eq:sdpconic} with $k{=}1$, $\ell{=}m_2{+}1$, $d{=}0$, and $n_2{=}\dots{=}n_\ell {=} 1$.
  Note that $r_{i+1}{=}1$ when the $i$-th inequality constraint is inactive, and is zero otherwise.
  Hence
  $m' = m - \#(\text{inactive constrs}) = \#(\text{active constrs})$.
  This is consistent with the results from \Cref{sec:inequality_constrained_sdp_s}.
\end{example}

\begin{example}[Second-order cone]
  Let $\mathcal{Q}^n := \{x \!\in\! \RR^n: \|(x_2,\dots,x_n)\| \!\leq\! x_1 \}$
  be the second-order cone.
  Consider minimizing a linear cost on $\SS_+^{n_1} \!\times\! \mathcal{Q}^{n_2}$ subject to $m_1$ linear equalities.
  Apply the Burer-Monteiro factorization to the matrix in~$\SS_+^{n_1}$.
  We can embed $\mathcal{Q}^{n_2}$ inside $\SS^{n_2}_+$ by adding $\tau(n_2{-}1)$ new linear equalities, see \cite[pg.7]{Alizadeh2003}.
  So this is a special case of~\eqref{eq:bmconic} with $\ell {=} 2$, $k{=}1$, $d{=}0$, $m {=} m_1 {+} \tau(n_2{-}1)$.
  Given $x \!\in\! \mathcal{Q}^{n_2}$,
  the rank of the corresponding PSD matrix is $r_2{=}0$ if $x{=}0$,
  $r_2{=}n_2{-}1$ if $x$ lies in the boundary,
  and $r_2{=}n_2$ if $x$ lies in the interior.
  So \Cref{thm:mainconic} applies when $\tau(p_1) \!>\! m_1 {+} \tau(n_2{-}1) {-} \tau(r_2)$,
  where $r_2$ is the smallest feasible rank.
  We point out that embedding $\mathcal{Q}^{n_2}$ inside~$\SS_+^{n_2}$ is used for the analysis,
  but we do not need to do this in practice.
  The reason is that the embedding preserves critical points.
\end{example}

We also provide an explicit characterization of the costs $C$ for which spurious 2-critical points may exist.
These costs lie in the Minkowski sum of two algebraic sets,
which are closely related to the ones in \eqref{eq:varietyrank} and~\eqref{eq:varietyL}.

\begin{theorem}\label{thm:sufficientconic}
  If~\eqref{eq:bmconic} has a spurious 2-critical point,
  then $C$ lies in the algebraic set
  $\, \underline{\mathcal V} \!\times\! \overline{\mathcal V} \!\times\! \{0^d\}+ \image \mathcal{A}^*\subset \SS^{\bf n} \!\times\! \RR^d$, with
  \begin{gather*}
    \underline{\mathcal V} \;:=\; \bigcup_{j\in [k]: \,p_j {\leq} n_j}
    (\SS^{n_1}\!\times\! \cdots \!\times\! \SS^{n_{j-1}}\!\times\! \SS^{n_j}_{n_j-p_j} \!\times\! \SS^{n_{j+1}} \!\times\! \cdots \!\times\! \SS^{n_\ell})
    \;\subset\; \SS^{\underline{\bf n}},
    \\
    \overline{\mathcal V} \;:=\; \bigcup_{r_{k+1},\dots,r_\ell}
    (\SS^{n_{k+1}}_{n_{k+1}-r_{k+1}}\!\times \cdots \times \SS^{n_\ell}_{n_\ell-r_\ell})
    \;\subset\; \SS^{\overline{\bf n}},
  \end{gather*}
  where the last union is over the possible ranks $r_{k+1},\dots,r_\ell$ in~\eqref{eq:sdpconic}.
\end{theorem}

\Cref{thm:mainconic} follows from \Cref{thm:sufficientconic} by counting dimensions.
Let $p_{\min}$ be the minimum of $p_1,\dots,p_k$,
ignoring the values with $p_j \!>\! n_j$.
Note that
\begin{gather*}
  \dim \underline{\mathcal V}
  = \sum_{j\leq k} \!\tau(n_j) - \tau(p_{\min}),
  \qquad
  \dim \overline{\mathcal V}
  = \max_{r_{k+1}\dots r_\ell} \sum_{j>k} \!\tau(n_j){-}\tau(r_j).
\end{gather*}
Let $D := \dim( \SS^{\bf n} \!\times\! \RR^d) = \tau(n_1){+}\cdots{+}\tau(n_\ell){+}d$.
If $\tau(p_{\min}) \!>\! m'$, then
\begin{gather*}
  \dim (\underline{\mathcal V}\!\times\! \overline{\mathcal V} \!\times\! \{0\} \!+\! \image \mathcal{A}^*)
  = m + \sum_{j\leq k} \!\tau(n_j) - \tau(p_{\min}) +
  \max_{r_{k+1}\dots r_\ell} \sum_{j>k} \!\tau(n_j){-}\tau(r_j)
  \\
  =\, D -\tau(p_{\min}) +
  \max_{r_{k+1},\dots,r_\ell} \bigl\{m\!-\!d\!-\!\sum_{j>k} \tau(r_j) \bigr\}
  \,=\, D -\tau(p_{\min}) + m' \;<\; D,
\end{gather*}
Hence $\underline{\mathcal V}\!\times\! \overline{\mathcal V} \!\times\! \{0\} \!+\! \image \mathcal{A}^*$ has measure zero.

We proceed to prove \Cref{thm:sufficientconic}.
We first derive the optimality conditions for~\eqref{eq:bmconic}.
This is a special instance of~\eqref{eq:nlpconic},
so we need to specialize~\eqref{eq:conditionsconic0}.
For $\lambda\!\in\!\RR^m$, consider the slack variable
$S(\lambda) := C {-} \mathcal{A}^{*}(\lambda) \in \SS^{\bf n}{\times}\RR^n$.
Let
$S_j(\lambda) \in \SS^{n_j}$
be the $j$-th component of~$S(\lambda)$.
Similarly define $\overline{S}(\lambda) \in \SS^{\overline{\bf n}}$ and $s(\lambda) \in \RR^{d}$.
The criticality conditions are:
\begin{subequations}\label{eq:conditionsconic}
\begin{gather}
  \label{eq:conditionconicfirst}
  (q(Y), \overline{X},x) \!\in\!\! \mathscr{X},
  \;\;\,
  \overline{S}(\lambda) \!\in\! \SS^{\overline{\bf n}}_+,
  \;\;\,
  \langle \overline{S}(\lambda), \overline{X}\rangle \!=\! 0,
  \;\;\,
  s(\lambda) \!=\! 0,
  \;\;\,
  S_j(\lambda)  Y_j \!=\! 0,
  \\
  \label{eq:conditionconicsecond}
  S_j(\lambda)\bullet U_j U_j^T \geq 0,
  \quad\forall\, U_j\!\in\! \RR^{n_j\times p_j} \text{ s.t. }
  \mathcal{A}_j(U_j Y_j^T) \!=\!0
  \quad(\text{for }j\!\in\! [k]).
\end{gather}
\end{subequations}

We now provide sufficient conditions for global optimality of critical points.

\begin{lemma} \label{thm:firstpartconic}
  Either of the following conditions imply global optimality:
  \begin{enumerate}[label=(\roman*)]
    \item
      $( Y,\overline{X},x)$ is 1-critical and $S_j(\lambda) \!\in\! \SS_+^{n_j}$ for $j \!\in\! [k]$.
    \item
      or $( Y,\overline{X},x)$ is 2-critical and $Y_j$ is column rank deficient for $j \!\in\! [k]$.
  \end{enumerate}
\end{lemma}
\begin{proof}
  The proof is analogous to \Cref{thm:firstpart}.
  For \textit{(i)} we compare~\eqref{eq:conditionconicfirst} with the primal/dual optimality conditions for \eqref{eq:sdpconic}.
  For \textit{(ii)} we use a vector $z_j \!\in\! \RR^{p_j}$ in the right kernel of $Y_j$ in order to show that $S_j(\lambda) \in \SS_+^{n_j}$.
\end{proof}

\begin{proof}[\proofmainconic]
  Let $( Y,\overline{X},x,\lambda)$ a spurious point satisfying~\eqref{eq:conditionsconic}.
  By \Cref{thm:firstpartconic}\textit{(ii)} we have tat $\rank  Y_j \!=\! p_j$ for some $j\!\in\![k]$.
  As $S_j(\lambda)  Y_j \!=\! 0$ then $ S_j(\lambda) \!\in\! \SS^{n_j}_{n_j-p_j}$.
  Let $(r_{k+1},\dots,r_\ell)$ be the ranks of $\overline{X}$.
  Since $\langle \overline{S}(\lambda), \overline{X}\rangle \!=\! 0$ and both lie in $\SS^{\overline{\bf n}}_+$,
  then $\overline{S}(\lambda) \subset \SS^{n_{k+1}}_{n_{k+1}-r_{k+1}}\!\!\times\! \cdots \!\times\! \SS^{n_\ell}_{n_\ell-r_\ell}$.
  Hence
  $S(\lambda)\in \underline{\mathcal V}\!\times\! \overline{\mathcal V}\!\times\!\{0\}$, as $s(\lambda)\!=\!0$.
  The result follows from
  $C = S(\lambda) \!+\! \mathcal{A}^*(\lambda)$.
  \addQED
\end{proof}

As illustrated next, \Cref{thm:sufficientconic} can be used even when~$C$ is not generic.

\begin{example}[Matrix sensing] \label{exmp:sensing2}
  We revisit the problem of sensing symmetric matrices from \Cref{exmp:sensing}.
  For $X\!\in\! \SS^n$, its nuclear norm satisfies:
  \begin{align*}
    \|X\|_* \quad=\quad
    \min_{Z} \;\; \id_n \bullet Z
    \;\;\text{ such that }\;\;
    Z\!+\!X\in \SS^n_+,\;\;
    Z\!-\!X\in \SS^n_+.
  \end{align*}
  Let $X_1 \!:=\! \frac{1}{2}(Z{+}X)$, $X_2 \!:=\! \frac{1}{2}(Z{-}X)$.
  We can rewrite problem~\eqref{eq:sensing} as follows:
  \begin{align*}
    \min_{X_1,X_2} \;\; \id_n {\bullet} X_1 \!+\! \id_n {\bullet} X_2
    \;\;\text{ such that }\;\;
    \mathcal{A}(X_1)\!-\!\mathcal{A}(X_2)\!=\!b,\;
    X_1\!\in\! \SS^n_+,\;
    X_2\!\in\! \SS^n_+.
  \end{align*}
  Consider the Burer-Monteiro method applied to both matrices $X_1,X_2$,
  so that $k\!=\!\ell\!=\!2$,
  using the same rank $p$ for both matrices.
  We will prove that there are no spurious 2-critical points
  when $\mathcal{A}:\SS^n\!\to\!\RR^m$ is generic and $\tau(p)\!>\!m$.
  By \Cref{thm:sufficientconic}, we need to show that
  \begin{align*}
    (\id_n,\,\id_n)
    \;\notin\;
    \underline{\mathcal V} \,+\,  (1,-1) {\otimes} \image \mathcal{A}^*,
    \quad\text{ where }\quad
    \underline{\mathcal V} \,:=\, \SS^n_{n-p} \!\times\! \SS^n \;\cup\; \SS^n \!\times\! \SS^n_{n-p}.
  \end{align*}
  It suffices to see that $\id_n \notin \SS^n_{n-p} \!+\! \image \mathcal{A}^*$.
  But this was shown in \Cref{thm:genericA}.
\end{example}

\subsection*{Acknowledgments}
The author thanks Nicolas Boumal, Ankur Moitra, Pablo Parrilo, and David Rosen for helpful discussions and comments.

\appendix
\section{Regularity with generic constraints}\label{sec:regularity}

In this section we prove \Cref{thm:licqgeneric}.
Our proof relies on Sard's theorem from differential geometry,
see e.g.,~\cite[\S2]{Milnor1997}.

\begin{theorem}[Sard]\label{thm:sard}
  Let $f: \RR^n \to \RR^m$ be a smooth map, with $n \!\geq\! m$.
  Let $v \!\in\! \RR^m$ be a generic point.
  Then $\rank(\nabla f(y)) \!=\! m$ for any $y \!\in\! f^{-1}(v)$.
\end{theorem}

\begin{proof}[\prooflicq]
  Fix a set of indices $I \!\subset\! [m]$, and let
  \begin{align*}
    \mathcal{M}_I := \{Y \in \RR^{n\times p}:
    A_i \bullet Y Y^T \!= b_i \text{ for } i \in I \}.
  \end{align*}
  We claim that \eqref{eq:licq} holds at all points on $\mathcal{M}_I$
  (i.e., the gradients are linearly independent).
  If this happens for each $I\!\subset\![m]$,
  then \eqref{eq:licq} also holds for the feasible set of~\eqref{eq:bm}.
  So it suffices to show the claim.

  We prove the claim under a more restrictive genericity setting.
  We assume that each $b_i \!\neq\! 0$ and that $A_i \!=\! \alpha_i \bar{A}_i$,
  where $\{\bar{A}_i\}$ are fixed matrices
  and $\{\alpha_i\}$ are generic scalars.
  Let
  \begin{align*}
    f_I : \RR^{n\times p} \to \RR^I,
    \qquad
    Y \,\mapsto\, \left( \bar{A}_i \bullet Y Y^T  : i \in I \right).
  \end{align*}
  The vector $v \!:=\! (b_i / \alpha_i : i \!\in\! I)$ is generic
  since $\{\alpha_i\}$ are generic.
  By \Cref{thm:sard}, $\nabla f_I(Y)$ is full rank for any $Y \!\in\! f_I^{-1}(v) = \mathcal{M}_I$.
  So \eqref{eq:licq} holds on~$\mathcal{M}_I$.
  \addQED
\end{proof}

\bibliographystyle{abbrv}
\bibliography{refs.bib}

\end{document}